\documentclass[12pt,fleqn]{article}
\usepackage{latexsym}
\usepackage{amsmath}
\usepackage{amssymb}   
\usepackage{epsfig}
\usepackage{amsfonts}

    \usepackage[english]{babel}

\usepackage[usenames, dvipsnames]{color}

\begin{document}
\newtheorem{definition}{Definition}[section] 
\newtheorem{theorem}[definition]{Theorem}
\newtheorem{lemma}[definition]{Lemma}
\newtheorem{proposition}[definition]{Proposition}
\newtheorem{examples}[definition]{Examples}
\newtheorem{corollary}[definition]{Corollary}
\def\square{\Box}
\newtheorem{remark}[definition]{Remark}
\newtheorem{remarks}[definition]{Remarks}
\newtheorem{exercise}[definition]{Exercise}
\newtheorem{example}[definition]{Example}
\newtheorem{observation}[definition]{Observation}
\newtheorem{observations}[definition]{Observations}
\newtheorem{algorithm}[definition]{Procedure}
\newtheorem{criterion}[definition]{Criterion}
\newtheorem{algcrit}[definition]{Algorithm and criterion}

\newenvironment{prf}[1]{\trivlist
\item[\hskip \labelsep{\it
#1.\hspace*{.3em}}]}{~\hspace{\fill}~$\square$\endtrivlist}
\newenvironment{proof}{\begin{prf}{Proof}}{\end{prf}}

\title{Linear differential equations with finite differential Galois group}
\author{M. van der Put, C. Sanabria Malag\'on, J. Top}
\date{\today}
\maketitle

\begin{abstract} For a differential operator $L$ of order $n$ over $C(z)$ with a finite (differential) Galois group 
$G\subset {\rm GL}(C^n)$, there is an algorithm, by M. van Hoeij and J.-A.~Weil, which computes the associated evaluation of the invariants $ev:C[X_1,\dots ,X_n]^G\rightarrow C(z)$.

The procedure proposed here does the opposite: it uses a theorem of E.~Compoint and computes the
operator $L$ from a given evaluation $h$. Moreover it solves a part of the inverse problem of producing $L$ for a given representation of a finite group $G$. 
Another part considered here, is finding irreducible $G$-invariant curves $Z\subset \mathbb{P}(C^n)$ with $Z/G$ of genus zero and constructing evaluations from this.  
The theory developed here is illustrated by various examples, and relates to and continues classical work of H.A.~Schwarz,  G.~Fano, F.~Klein and  
A.~Hurwitz.
 \footnote{ key words: differential Galois theory, inverse problem, invariant curves, Schwarz maps, evaluation of invariants}  \end{abstract}

\section{\rm Introduction}  Let $k$ denote the differential field $C(z)$ with derivation $\frac{d}{dz}$  where $C$ is an algebraically closed field 
 of characteristic zero. The algebraic closure of $k$ is denoted by $\overline{k}$ and we write $\pi$ for the Galois group of $\overline{k}/k$.  

The category ${\rm Diff}_{\overline{k}/k}$ that we study
here, has as objects the finite dimensional differential modules $M$ over $k$ which become trivial over the field $\overline{k}$. This condition on $M$ is equivalent to
$M$ having a finite differential Galois group. The morphisms in this category are the $k$-linear maps that commute with differentiation.

 Let $Repr_\pi$ denote the category of the (continuous) representations of
$\pi$ on finite dimensional $C$-vector spaces. The functor ${\rm Diff}_{\overline{k}/k}\rightarrow Repr_\pi$, which associates to a differential module
$M$ its solution space $\ker (\partial ,\overline{k}\otimes_k M)$, is known to be an equivalence of (Tannakian) categories. It is also known
that every finite group $G$ is a continuous image of $\pi$. \\

The aim of this paper is to make this equivalence of categories explicit for special cases. There are two directions to consider:\\
(i)  Compute a differential operator connected to  a given representation of a given finite group and some additional data.\\
(ii)  Construct the Picard--Vessiot field for a given module $M\in {\rm Diff}_{\overline{k}/k}$, when $M$ is represented by a differential operator $L$.\\

For order $n=2$, the Schwarz' list (compare \cite{vdP-U} for a modern version) and Klein's theorem (e.g., discussed
in  \cite{B-D} and in \cite{Ber})  are classical results for (i). J.~Kovacic' paper \cite{Ko} initiated (ii). \\
In the case $n=3$ 
Hurwitz' paper \cite{H}  produces examples for (i). This method was refined in \cite{vdP-U}.  Klein's theorem is generalized in, e.g.,\  \cite{Ber, S1,S2,S3}.  For (ii) there are many papers \cite{S-U,Ho, H-W}. \\
If $n$ is general, work of Compoint and Singer \cite{C, C-S} considers (ii). Not much seems to have been done concerning (i).\\

 The present paper is mainly concerned with (i). It clarifies and extends work of \cite{Ber, S1,S2,S3}.
 Moreover we propose new methods and examples for (i) related to a theorem of Compoint \cite{C, B} concerning invariants.

The paper is organized as follows.\\
Section~\ref{two} introduces Picard-Vessiot curves, evaluation of invariants, Fano curves, and Schwarz maps.\\
Section~\ref{three} constructs, given a finite irreducible subgroup $G\subset \mbox{GL}(C^n)$ and a $G$-invariant
curve  $Z\subset \mathbb{P}(C^n)$, a `standard' differential operator
over $C(z)$ of order $n$ with projective Galois group equal to the image of $G$ in $\mbox{PGL}(C^n)$ and Fano curve $Z$.\\
Section~\ref{four} gives a formulation of Compoint's theorem and uses this to describe an algorithm computing
a standard differential operator.\\
Section~\ref{five} contains a proof showing that for $n=2$ our algorithm results in the known standard operators.
Moreover it illustrates the theory by providing various examples such as the subgroups $G_{168}$, $H_{72}$, and $A_5$ 
 of $\mbox{GL}(C^3)$.

\section{ \rm Objects associated to a differential operator $L$ over $k=C(z)$ with finite differential Galois group}\label{two}

$L$ has the form $d_z^n+a_{n-1}d_z^{n-1}+\cdots +a_0$ with all $a_i\in {C}(z)$, $d_z=\frac{d}{dz}$ and all solutions are
supposed to be algebraic over ${C}(z)$. Associated to  $L$ is:\\

\noindent (1) {\it The  Picard--Vessiot field} $K\supset {C}(z)$ with its Galois group $G$.\\

\noindent (2) {\it The (contravariant)  solution space} $V\subset K$ of $L$ with the action of $G$ on
it; the image of $G\subset {\rm GL}(V)$ into ${\rm PGL}(V)$ will be denoted by $G^p$.\\

\noindent (3) {\it The  Picard--Vessiot curve}  $X_{pv}$ (over ${C}$) with function field $K$ provided
with the action of $G$ and an isomorphism $X_{pv}/G\cong \mathbb{P}^1_z$.\\
 By $\mathbb{P}^1_z$ we denote the projective line with function field ${C}(z)$.\\

\noindent (4)  {\it Evaluation of the invariants}.  One considers a $C$-linear homomorphism $\phi :C[X_1,\dots ,X_n]\rightarrow K$ which sends the 
variables $X_1,\dots ,X_n$ to a basis of $V$. The $C$-linear action of $G$ on $C[X_1,\dots ,X_n]$
is defined by the $G$-invariance of $CX_1+\cdots +CX_n$ and the $G$-equivariance of $\phi$.  
This makes $G$ into a subgroup of ${\rm GL}(n,C)$.  The homomorphism $\phi$ induces a homomorphism 
$ev: C[X_1,\dots ,X_n]^G\rightarrow K^G=C(z)$ which we will call the {\it evaluation  of the invariants}.
In particular, $C[X_1,\dots ,X_n]^G=C[f_1,\dots ,f_N]$ where $f_1,\dots ,f_N$ are homogeneous generators
and $ev$ maps each $f_i$ to an element in $C(z)$. \\

  Now suppose that the action of $G$ on $V$ is known and  is  irreducible. If we define the action of $G$ on $CX_1+\cdots +CX_n$
   such that an equivariant $\phi$ with $\phi (CX_1+\cdots +CX_n)=V$  exists, then this $\phi$ is unique up to multiplication by a scalar
    $c \in C^*$.  As a consequence, {\it the evaluation map is
   unique} up to changing each $ev(f_i)$ into $c ^{\deg f_i}ev(f_i)$.\\

\noindent (5) {\it The  Fano curve}. \\
 $H\subset \ker (\phi )$, the homogeneous kernel, is the ideal generated generated by the homogeneous elements in $\ker( \phi )$.
 For $n=2$ one has $H=0$. For notational reasons we will call $\mathbb{P}(V)=\mathbb{P}^1$ itself {\it the Fano curve} in this case.\\
 
 Suppose that $n>2$, then $H$ defines an irreducible curve in $\mathbb{P}^{n-1}$, invariant under the action of $G$.
Indeed, $H$ is the homogeneous ideal induced by the kernel $J$ of the corresponding homomorphism
 ${C}[\frac{X_2}{X_1},\dots ,\frac{X_n}{X_1}]\rightarrow
 K$. It is a curve since  $K/{C}$ has transcendence degree 1. The curve in
 $\mathbb{P}^{n-1}$ defined by $H$ will be denoted by $X_{fano}$ and
 will be called {\it the Fano curve}. This curve was indeed considered by
 Fano in his 1900-paper \cite{Fa}.\\

\noindent (6) {\it The  Schwarz map}. The homomorphism  ${C}[X_1,\dots ,X_n]/H\rightarrow K$ induces
a morphism of curves $Schw: X_{pv}\rightarrow X_{fano}$ which is $G$-equivariant. After dividing
by $G$ we obtain a multivalued map $\mathbb{P}^1_z=X_{pv}/G\dots
\rightarrow X_{fano}$.  Now $Schw$ is in fact {\it the Schwarz map}, known from the case $n=2$.

Dividing $Schw$ by $G$ we obtain $qSchw: \mathbb{P}^1_z=X_{pv}/G\rightarrow
X_{fano}/G^p$ which can be called {\it the quotient Schwarz map}.
 
The relation between $X_{pv}$ and $X_{fano}$ is in general not
obvious.  We will prove the following statement:

\begin{lemma}
 Suppose that  $qSchw: \mathbb{P}^1_z=X_{pv}/G\rightarrow
X_{fano}/G^p$ is birational. The subgroup $c(G)$ of $G\subset
{\rm GL}(V)$ consists of multiples of the identity. Then
$c(G)$ acts trivially on the curve $X_{fano}$ and therefore
$X_{pv}\rightarrow X_{fano}$ factors over $X_{pv}/c(G)$. 
 The morphism $X_{pv}/c(G)\rightarrow X_{fano}$ is birational. \end{lemma}

\begin{proof} Let $x_1,\dots ,x_n\in K$ be a basis of $V\subset
K$. Then $K={C}(z)[x_1,\dots ,x_n]$, 
$K^{c(G)}={C}(z)[\frac{x_2}{x_1},\dots ,\frac{x_n}{x_1}]$ 
and ${C}(X_{fano})$ is the field of fractions of the algebra 
  ${C}[\frac{x_2}{x_1},\dots ,\frac{x_n}{x_1}]$.
Since $C(z)={C}(X_{fano})$ one has
$K^{c(G)}={C}(X_{fano})$.\end{proof}

\begin{remark}{\rm $\ $ \\
 In \S 3 we define a {\it standard operator} $L_{st}$  over $k=C(z)$ associated to a finite group
$G\subset {\rm GL}(C^n)$ and a $G$-invariant irreducible curve $X\subset \mathbb{P}(C^n)$ such that (the normalisation of) $X/G$ has genus zero. 
This operator $L_{st}$ has  the property that $X_{fano}\simeq X$ (i.e., $X$ and $X_{fano}$ are birational) and that 
$qSchw: \mathbb{P}^1_z\rightarrow X_{fano}/G^p$ is birational.

 For a general linear differential equation $L$ with differential Galois group $G$ and $X_{fano}\simeq X$, the normalization of 
 $X_{fano}/G^p$ is $\mathbb{P}^1_s$ and $L$ is obtained from a standard operator defined in terms of $s$ by a weak pullback 
  (see \S 3 for the definition) along the morphism $\mathbb{P}^1_z\rightarrow \mathbb{P}^1_s$. 


 
 }\end{remark}

\section{\rm Construction of a standard operator and pullbacks}\label{three}

{\it The given data are}:\\
 A vector space $V$ over $C$ of dimension $n\geq 2$ and an {\it irreducible} finite group $G\subset {\rm GL}(V)$.
The image of $G$ into ${\rm PGL}(V)$ is called $G^p$.  \\

\noindent (i) {\bf For $n=2$},  there are no additional data. For notational reasons, one writes $Z=\mathbb{P}(V)$. Further $Z/G^p$ is identified
with $\mathbb{P}^1_z$ and so $C(Z/G^p)=C(z)$. One knows that $G^p\in \{D_n,A_4,S_4,A_5\}$ and the variable $z$ is chosen such that $z=0,1,\infty$ are the branch of $Z\rightarrow Z/G^p$.\\

\noindent (ii) {\bf For $n>2$}, the additional data are a $G^p$-invariant, irreducible curve $Z\subset \mathbb{P}(V)$ such that the normalisation of $Z/G^p$ has genus 0. The function field $C(Z/G^p)$ of $Z/G^p$ is identified with $C(z)$.\\

\begin{observation} {\rm The standard operator that we will construct depends only on $G^p$ and $Z$. Therefore we may suppose
that $G$ lies in ${\rm SL}(V)$ and maps surjectively to $G^p$ and that no proper subgroup of $G$ has this property. 
In particular,
the kernel of $G\rightarrow G^p$ has the form $\{\lambda \cdot {\bf 1} |\ \lambda ^m=1\}$ 
for a certain divisor $m$ of $n$.} \hfill $\square$ \end{observation}

The above data yield inclusions $C(z)=C(Z)^{G^p}\subset C(Z)\subset \overline{C(z)}$, where $z$ is unique up to a fractional linear transformation
and the embedding $C(Z)\subset \overline{C(z)}$ is unique up to an automorphism of $C(Z)$, i.e., an element of $G^p$.

We would like to identify $V$ with the solution
space in $\overline{C(z)}$ of the standard operator to be constructed. However, $V$ does not lie in $C(Z)$. 

One chooses $\ell \in V,\ \ell \neq 0$, then $\frac{V}{\ell}$ belongs to $C(Z)$. The vector space $\frac{V}{\ell}$ is not invariant under $G^p$, 
or what is the same, it is not invariant under $\pi$.\\

\begin{lemma} There exists an element $f\in \overline{C(z)}^*$ such that $f\frac{V}{\ell}$ is invariant under $\pi$.
The canonical map $\mathbb{P}(V)\rightarrow \mathbb{P}(f\frac{V}{\ell})$, given by $v\mapsto f\cdot \frac{v}{\ell}$ is equivariant
for the action of $\pi$.  \end{lemma}
\begin{proof} The group $G$ is supposed to have the properties of Observation 3.1. For each $\sigma \in G^p$, one denotes by
$\tilde{\sigma}$ an element in $G$ with image $\sigma$. Now $\sigma (\frac{V}{\ell })=\frac{V}{\tilde{\sigma}\ell}=
\frac{\ell}{\tilde{\sigma}\ell}\cdot \frac{V}{\ell}$. The term $\frac{\ell}{\tilde{\sigma}\ell}$ depends in general on the choice of 
$\tilde{\sigma}$. But $(\frac{\ell}{\tilde{\sigma}\ell})^m$ depends only on $\sigma$ and 
$\sigma \mapsto (\frac{\ell}{\tilde{\sigma}\ell})^m$ is a 1-cocycle. By Hilbert 90, there is an element $f\in C(Z)$ such that
$\frac{\sigma f}{f}\cdot (\frac{\ell}{\tilde{\sigma}\ell})^m=1$ for all $\sigma \in G^p$. \\

For the case $m=1$ we conclude that $f\frac{V}{\ell}\subset C(Z)$ is invariant under $G^p$ (and thus also under $\pi$). 
For the case $m>1$ we {\it claim} that the equation $T^m-f$ is irreducible over $C(Z)$. The field $C(Z)(f_m)$ with $f_m^m=f$
is a Galois extension of $C(Z)$ since for every $\sigma \in G^p$ one has $\frac{\sigma f}{f}$ is an $m$th power in $C(Z)$.
We may embed $C(Z)(f_m)$ into $\overline{C(z)}$ and conclude that $f_m\frac{V}{\ell}$ is invariant under $\pi$.\\
 
{\it Now we prove the claim}. If the equation $T^m-f$ is reducible over $C(Z)$, then there exists a proper divisor $d$ of $m$ and an
element $g\in C(Z)$ with $g^d=f$. The expression $E:=\frac{\sigma g}{g}\cdot (\frac{\ell}{\tilde{\sigma}\ell})^{m/d}$ has the property
that $E^d=1$. One can consider for each $\sigma \in G^p$ the elements $\tilde{\sigma}\in G$ such that $E=1$. This defines a proper
subgroup $H$ of $G$ which has image $G^p$.  This contradicts the assumptions on $G$.

The last statement of the lemma follows from 
$\sigma (f\frac{v}{\ell})=\frac{\sigma f}{f}\cdot \frac{\ell}{\sigma \ell}\cdot f\frac{\sigma v}{\ell}$.
\end{proof}

The monic operator $L$ of order $n$ over $\overline{C(z)}$, defined by  $\ker (L,\overline{C(z)})=W:=f\cdot \frac{V}{\ell}$ has its coefficients in $C(z)$, since $W$ is invariant under $\pi$. This operator $L$ is not yet unique since we have made choices for $\ell$ and $f$.\\

{\it The standard operator $L_{st}$} is defined to be the operator of the form $L_{st}=(\frac{d}{dz})^n+0\cdot (\frac{d}{dz})^{n-1}+\cdots$,
obtained from the above $L$ by a shift $\frac{d}{dz}\mapsto \frac{d}{dz}+a$ for suitable $a=\frac{h'}{h}$ with 
$h\in \overline{C(z)}^*$.

\begin{corollary} $L_{st}$ does not depend on the choices of $\ell$ and $f$ in Lemma {\rm 3.2}.
The solution space of $L_{st}$ has the form $g\cdot W$ for certain $g\in \overline{C(z)}^*$.\\
The differential Galois group $H$ of $L_{st}$  lies in ${\rm SL}(g\cdot W)$.\\
Let the projective spaces $\mathbb{P}(V)$ and $\mathbb{P}(g\cdot W)$  be canonically identified,  
 then $H^p=G^p$ and the Fano curve of $L_{st}$ is equal to $Z$.
 \end{corollary}

The first statement of the corollary follows easily from Lemma 3.4 and Observation 3.5. The other statements follow from
 the construction of $L_{st}$.

\begin{lemma} Let $L_1,L_2$ be monic differential operators over $C(z)$ such that all their solutions are algebraic.
Let $V_1,V_2\subset \overline{C(z)}$ denote the two solution spaces.  The following are equivalent:\\
(a). $L_1$ is obtained  from $L_2$ by a shift $\frac{d}{dz}\mapsto \frac{d}{dz}+a$ for some element $a\in {C}(z)$.\\
(b). There exists $f\in   \overline{{C}(z)}^*$ such that $V_2=fV_1$.
\end{lemma}
\begin{proof} (a)$\Leftrightarrow$(b). Let $L_1$ be obtained from $L_2$ by the shift 
$\frac{d}{dz}\mapsto \frac{d}{dz}+a$. One writes $a=\frac{f'}{f}$ with $f$ in some differential field containing $\overline{C(z)}$.
 One finds $V_2=fV_1$. Since $V_1,V_2\subset \overline{C(z)}$ one actually has $f\in \overline{C(z)}^*$.\\
If $V_2=fV_1$, then clearly $L_1=f^{-1}\circ L_2\circ f$. Since $f^{-1}\circ \frac{d}{dz}\circ f=\frac{d}{dz}+\frac{f'}{f}$, one has that $L_1$ is obtained from $L_2$ by the shift $\frac{d}{dz}\mapsto \frac{d}{dz}+\frac{f'}{f}$. Note that $\frac{f'}{f}\in C(z)$
since $L_1$ and $L_2$ are defined over $C(z)$.\end{proof}

\begin{observation}{\rm (1). For general monic differential operators $L_1,L_2$ of order $n$, property (a) of Lemma 3.4 is called
{\it projective equivalence}.  If both $L_1$ and $L_2$ have the form $d_z^n+0.d_z ^{n-1}+\cdots $, then projective equivalence 
implies equality. \\
(2). The implication (b)$\Rightarrow$(a) in Lemma 3.4 holds for general differential operators. However (a)$\Rightarrow$(b) is in general false since
the equation ${f'}=af$ with  $a\in C(z)$, need not have a solution on $\overline{C(z)}^*$.\\
(3). For differential modules $M_1,M_2$ there is a somewhat different notion of projective equivalence defined by:
there is a 1-dimensional module $E$ such that $M_1\otimes E\cong M_2$.\\
(4). {\it Projective equivalence} of subgroups $G_1,G_2\subset {\rm GL}(V)$ means that $G_1^p=G_2^p\subset {\rm PGL}(V)$.
Projective equivalence of operators imply projective equivalence of their differential Galois groups but the converse is false.
} \hfill $\square$ \end{observation}

\begin{definition}
{\rm Consider a homomorphism
$\phi: {C}(z)[\frac{d}{dz}]\rightarrow
{C}(x)[\frac{d}{dx}]$ of the form: $z\mapsto \phi (z)\in
{C}(x)\setminus {C}$ and $\frac{d}{dz}\mapsto
\frac{1}{\phi(z)'}(\frac{d}{dx}+b)$ with $b\in {C}(x)$. Let $L\in
{C}(z)[\frac{d}{dz}]$. A {\it weak pullback} of $L$ is an operator
of the form $a\cdot \phi (L)$ with $a\in {C}(x)^*$. The restriction
of $\phi$ to ${C}(z)\rightarrow {C}(x)$ is called the
pullback function.
}\end{definition}

\begin{proposition} Let  $L\in {C}(s)[\frac{d}{ds}]$ be a monic operator of order $n$ such that all solutions are algebraic.
$M\subset \overline{C(z)}$ denotes its solution space. 

The differential Galois group $G$ of $L$ is the image of 
$\pi\rightarrow {\rm GL}(M)$ and $L$ determines some $X_{fano}\subset \mathbb{P}(M)$. Then $L$ is a weak pullback
of the standard operator $L_{st}$ determined by the data $G^p$ and $Z=X_{fano}$. \end{proposition}    

\begin{proof} 
The standard operator $L_{st}$ has solution space $g\cdot W$ for some $g\in \overline{C(z)}^*$, where
$W=f\frac{M}{m}$ for suitable $f\in \overline{C(z)}^*$ and $m\in M,\ m\neq 0$. Further $C(X_{fano}/G^p)=C(z)$ is a subfield
of $C(s)$. This inclusion determines the pullback function. Use now 3.4 and 3.6. \end{proof}

\begin{remark}{\rm A standard differential equation for given $G^p$ and Fano curve $Z\subset \mathbb{P}(V)$ can be a proper
pullback of another standard equation. This occurs essentially only when $G^p$ is a proper subgroup
 of a finite automorphism group $H$ of  (the desingularisation of) $Z$.  }\end{remark}

\begin{examples}{\rm  For $n=2$ the groups $G^p$ are $D_n,A_4,S_4,A_5$ and the standard equations are well known 
\cite{B-D,Ber} et al. to be
\[(\frac{d}{dz})^2+\frac{3}{16z^2}+\frac{3}{16(z-1)^2}-\frac{n^2+2}{8n^2z(z-1)}\  \mbox{ for } D_n,      \] 
\[(\frac{d}{dz})^2+\frac{3}{16z^2}+\frac{2}{9(z-1)^2}-\frac{3}{16z(z-1)}   \ \ \  \ \mbox{ for } A_4,  \] 
\[(\frac{d}{dz})^2+\frac{3}{16z^2}+\frac{2}{9(z-1)^2}-\frac{101}{576z(z-1)} \ \ \ \mbox{ for } S_4,      \] 
\[(\frac{d}{dz})^2+\frac{3}{16z^2}+\frac{2}{9(z-1)^2}-\frac{611}{3600z(z-1)} \ \  \mbox{ for } A_5. \] 
}\end{examples}

\begin{remarks} { \rm A calculation of the standard operator $L_{st}$, using the above construction, is possible.
One has to compute the $f$ in Lemma 3.2 and one has to compute the derivation on $C(Z)[f]$ in order to compute the monic
differential operator $L$ with solution space $f\frac{V}{\ell}\subset C(Z)[f]$. Further a computation of a generator of 
$C(Z)^{G^p}$ is needed. However for the case $n=2$ the calculation is rather easy. We illustrate this for  the case $A_4\subset {\rm PSL}_2$ and its preimage 
$A_4^{SL_2}$ in ${\rm SL}_2$.\\

For $Z=\mathbb{P}^1$ we use homogeneous coordinates $x,y$ and the function field is $C(t)$ with $t=\frac{y}{x}$. The invariants 
(see \cite{Ca}) under the action of $A_4^{SL_2}$ are generated by:\\
$Q_3=xy(x^4-y^4)$,
 $Q_4=(x^4+\sqrt{-12}x^2y^2+y^4)\cdot  (x^4-\sqrt{-12}x^2y^2+y^4)$,
$Q_6= (x^4+\sqrt{-12}x^2y^2+y^4)^3+(x^4-\sqrt{-12}x^2y^2+y^4)^3$.
There is one relation $Q_6^2-Q_3^4-4Q_4^3=0$. The field of the homogeneous invariants of degree zero is generated over $C$ by
$\frac{Q_6}{Q_3^2}$ and $\frac{Q_4^3}{Q_3^4}$ and there is one relation $(\frac{Q_6}{Q_3^2})^2=1+4\frac{Q_4^3}{Q_3^4}$. Hence we can take
$z=\frac{Q_6}{Q_3^2}$ where $x,y$ in this expression is replaced by $x,tx$. This expresses $z$ as rational function in $t$
of degree 12. Thus $\frac{dt}{dz}$ is also known.  

Now $V=Cx+Cy$, take $\ell =x$, then $\frac{V}{\ell}=C1+Ct$. Then $f\in C(t)$ should satisfy 
$(\frac{x}{\tilde{\sigma}x})^2=\frac{f}{\sigma f}$. A convenient choice for $f$ is $\frac{1}{t'}$ where $t':=\frac{dt}{dz}$.
Thus the Picard--Vessiot field is $C(t)[  \sqrt{  t' }]$. The operator that we want to compute has solution space 
$C\frac{1}{     \sqrt{  t'  }    }+C\frac{t}{      \sqrt{t'  }      }$. This easily leads to the standard operator for case $A_4$.
The other operators of Examples 3.8. are computed in a similar way.  
 The case $n>2$ is more involved and we will use a new method  (see \S 4.3) using evaluations instead.  }  \end{remarks}

\begin{observations} The singular points of the standard equations.\\
 {\rm Instead of developing the general case we consider the case $n=3$, $G=G^p$ and $Z\subset \mathbb{P}(V)$, irreducible,
 $G$-invariant and $C(Z)^G=C(z)$. 
 
Let  $L=d_z^3+a_2d_z^2+a_1d_z+a_0$ be the operator with solution space 
 $W:= f\frac{V}{\ell}\subset C(Z)$ (see Lemma 3.2). $\tilde{Z}$ denotes the desingularisation of $Z$. Consider the Galois covering
 $\tilde{Z}\rightarrow \mathbb{P}_z^1$. \\ 

 Let $z_0\in \tilde{Z}$ be unramified. We may suppose that $z_0$ lies above $z=0$ in $\mathbb{P}_z^1$.
Then $\widehat{O}_{\tilde{Z},z_0}$ can be identified with $\mathbb{C}[[z]]$. The image of $W$ in $\mathbb{C}((z))$ has a basis 
$b_1, b_2, b_3$ with orders $n_1<n_2<n_3$. The equations $L(b_1)=L(b_2)=L(b_3)=0$ determine values of $a_2,a_1,a_0$  
 in $\mathbb{C}((z))$. For the case $(n_1,n_2,n_3)=(0,1,2)$ the operator $L$ has no singularity. For other cases one finds
 an apparent singularity. Typical example:\\
 $b_1=z^{n_1}, b_2=z^{n_2},b_3=z^{n_3}$ yields the equations 
 \[n_i(n_i-1)(n_i-2)=n_i(n_i-1)za_2+n_iz^2a_1+z^3a_0\mbox{ for } i=1,2,3.\]
 Then $za_2,z^2a_1,z^3a_0\in \mathbb{C}$. For $(n_1,n_2,n_3)=(0,1,2)$ one finds $a_2=a_1=a_0=0$;
 for $(n_1,n_2,n_3)=(0,1,3)$ one finds $a_2=z^{-1}, \ a_1=a_0=0$.  In general $n_1,n_2,n_3$ are the local exponents at
 the point $z=0$. \\

 Suppose that $z_0\in \tilde{Z}$ is ramified with ramification index $e$. Again we assume for notational convenience
 that $z_0$ maps to $z=0$.  Then $\widehat{O}_{\tilde{Z},z_0}=\mathbb{C}[[z^{1/e}]]$.
 The image of $W$ in $\mathbb{C}((z^{1/e}))$ generates this field over $\mathbb{C}((z))$. The space $W$ has a basis
 with orders $n_1/e,n_2/e,n_3/e$ and $n_1<n_2<n_3$ with $g.c.d. (n_1,n_2,n_3,e)=1$. It follows that $z=0$ has  a regular singularity
 with local exponents  $n_1/e,n_2/e,n_3/e$. We conclude:\\
 
 \noindent {\it The singular points for $L$ and its normalisation $L_{st}$, which are not apparent, are precisely 
  the branch points for $\tilde{Z}\rightarrow \mathbb{P}_z^1$}.\\
  
 The case $n=3$ and $G\subset {\rm SL}_3$ minimal such that the image is the given $G^p$ but $G\neq G^p$ (see the proof of Lemma 3.2), can be done in a similar way but replacing the field $C(Z)$ by $C(Z)(f_3)$ with Galois group $G$.  } \end{observations}

\section{\rm Compoint's theorem and evaluation of invariants}\label{four}  

\begin{theorem}[E. Compoint 1998] { Notation and assumptions:} \\ {\rm
 Suppose that the differential equation $y'=Ay$ over $C(z)$ has a reductive differential Galois group 
 $G\subset {\rm GL}_n(C)$. The differential algebra $R:=k[\{X_{i,j}\},\frac{1}{D}]$ 
 {\rm(}with  $D=\det (X_{i,j})$ {\rm )}  is defined by $(X'_{i,j})=A\cdot (X_{i,j})$. 
 
 \noindent Let $I$ be a maximal differential ideal in $R$ and $K$ the Picard--Vessiot field 
 obtained as field of fractions of $R/I$.\\

 \noindent $ {\rm GL}_n({C})$ acts on the $C(z)$-algebra $R$ by sending the matrix of variables $(X_{i,j})$
 to the matrix $(X_{i,j})\cdot g$ for any $g\in {\rm GL}_n({C})$. Then $G$ is identified with the $g\in {\rm GL}_n({C})$ such 
 that  $gI=I$. \\
 
\noindent  The algebra of invariants ${C}[\{X_{i,j}\}]^G$ is generated over ${C}$ by homogeneous elements
 $f_1,\dots ,f_N$ {\rm (}since $G$ is reductive{\rm )}. The natural map $R\rightarrow K$ induces a 
 homomorphism $ev_e: {C}[\{X_{i,j}\}]^G\rightarrow C(z)$ which is called the {\em evaluation of the invariants}. }\\
 
 \noindent {\bf The statement is}: $I=(f _1-ev_e(f_1),\dots ,f_N-ev_e(f_N))$. \end{theorem}

The proof of Compoint's theorem, \cite{C},  has been simplified in \cite{B} and Theorem 4.1 is almost identical to
the formulation in \cite{B}. We will apply Compoint's theorem for the case of finite differential Galois groups. 
Moreover we will need a formulation in terms of differential operators (or scalar differential equations).

\begin{corollary} {\rm Let $ L=d_z^n+a_{n-1}d_z^{n-1}+\cdots +a_1d_z+a_0$  over  $C(z)$ have a
finite differential Galois group $G$ and Picard--Vessiot field $K\subset \overline{C(z)}$.\\
Consider the homomorphism $\phi: R_0=C(z)[X_1,\dots ,X_n]\rightarrow K$ which sends $X_1,\dots ,X_n$
to a basis of the solution space of $L$ in $K$. $G$ acts $C(z)$-linear on $R_0$ by a $C$-linear action on
$CX_1+\cdots +CX_n$ which coincides with the action of $G$ (or of $\pi$) on the solution space of $L$.  \\

\noindent 
The restriction of $\phi$ to $C[X_1,\dots X_n]^G\rightarrow C(z)$ is also called {\em the evaluation of the invariants} and denoted by $ev$ (see also \S 2).
Write $C[X_1,\dots ,X_n]^G=C[\phi_1,\dots ,\phi_r]$ for certain homogeneous elements $\phi_k$. }{ \bf The statement is:\\

 $\ker (\phi )$ is generated by $\{ \phi _k-ev(\phi_k)|\ k=1,\dots ,r\}$. } \end{corollary}

\begin{proof} Write again $R_0=C(z)[X_1,\dots ,X_n]$ and $R:=C(z)[\{X_i^j\}_{i=1,\dots ,n}^{ j=0,\dots ,n-1}]$ where $X_i^j$ denotes formally
the $j$th derivative of $X_i$ (all $i,j$). The map $\phi: R_0\rightarrow K$ has a unique extension 
$\phi_e: R\rightarrow K$ defined by $\phi_e(X_i^j)=\phi(X_i)^{(j)}$ (all $i,j$). The restriction of $\phi$ to 
$R_0^G\rightarrow C(z)$ is called $ev$ and the restriction of $\phi_e$ to $R^G\rightarrow C(z)$ is called $ev_e$. \\

By Compoint's theorem, the ideal  $\ker (\phi_e)\subset R$ is generated by the set $\{F-ev_e(F)\ |\ F\in R^G\}$. We want to prove that the ideal $\ker(\phi )\subset R_0$ is generated by $\{F-ev(F)\ | \ F\in R_0^G\}$. We will construct a $C(z)$-algebra homomorphism $\Psi :R\rightarrow R_0$  which has the following properties:\\
$\Psi (r)=r$ for $r\in R_0$; $\Psi (X_i^0)=X_i$; $\phi \circ \Psi =\phi_e$ and $\Psi$ is $G$-equivariant.  \\

Consider an element $\xi \in \ker \phi$. Then also $\xi \in \ker \phi_e$ and $\xi$ is a finite sum
$\sum c(F)\cdot (F-ev_e(F))$ with $F\in R^G$ and $c(F)\in R$. Applying $\Psi$ to this expression yields 
$\xi =\sum \Psi (c(F))\cdot (\Psi(F)-\Psi(ev_e(F))$. Since $\Psi$ is $G$-equivariant  $\Psi(F)\in R_0^G$. Moreover 
$\Psi (ev_e(F))=ev(\Psi (F))$. This implies that $\xi$ lies in the ideal generated by the $\{F-ev(F)\ |\ F\in R_0^G\}$ 
in the ring $R_0$.\\

\noindent {\it Construction of $\Psi$}. Define a $C$-linear derivation $E:R_0\rightarrow R_0$ by $E(z)=1$ and, for
 $i=1,\dots ,n$, ${E}(X_i)\in R_0$ has the property that $\phi {E}(X_i)=\phi(X_i)'$. We note that  $E$ exists since the map $\phi:R_0\rightarrow K$ is surjective. Then $D:=\frac{1}{\#G}\sum _{g\in G}gEg^{-1}:R_0\rightarrow R_0$ is a $C$-linear derivation
 with $D(z)=1$, $\phi (D(X_i))=\phi(X_i)'$ for all $i$ and $D$ is $G$-equivariant.\\ 

Define the $C(z)$-algebra homomorphism $\Psi:R\rightarrow R_0$ by
$\Psi (X_i^j)=D^j(X_i)$ for all $i,j$. The first two properties of $\Psi$ are obvious.
Further $\phi (\Psi (X_i^j))=\phi (D^j(X_i))=\psi(X_i)^{(j)}$ (for all $i,j$) and so
$\phi \circ \Psi=\phi_e$. Finally $\Psi$ is $G$-equivariant   because $D$ is $G$-equivariant and the actions of $G$ on the vector spaces $CX_1^j+\cdots +CX_n^j$, for $j=0,\dots ,n-1$, are identical.  \end{proof}

\begin{algorithm} 
Constructing the differential operator from an evaluation.\\ 
{\rm Let an irreducible  finite group $G\subset {\rm GL}_n(C)$ be given. The group $G$ acts in the obvious way 
on $C[X_1,\dots ,X_n]$. Suppose that $C[X_1,\dots ,X_n]^G=C[f_1,\dots ,f_N]$ with known homogeneous $f_1,\dots ,f_N$.

Consider a $C$-algebra homomorphism $h:C[X_1,\dots ,X_n]^G\rightarrow C(z)$ such that the image of $h$ generates 
the field $C(z)$ over $C$. We will call such $h$ again {\it an evaluation of the invariants}. The aim is to compute a
differential operator $L=d_z^n+a_{n-1}d_z^{n-1}+\cdots +a_1d_z+a_0$ over $C(z)$ that induces the group $G$ and has evaluation in Corollary 4.2 equal to $h$. \\

The $C(z)$-algebra $R:=C(z)[x_1,\dots ,x_n]=C(z)[X_1,\dots ,X_n]/I$ where  $I=(f_1-h(f_1),\dots ,f_N-h(f_N))$ has finite dimension over $C(z)$.\\

(a). {\it Let us suppose that $R$ is reduced}. Then $R$ is a product of finite field extensions of $C(z)$. The derivation $\frac{d}{dz}$ has a unique extension to $R$ which we call $\tilde{D}$. We want to produce an operator 
$\tilde{L}:=\tilde{D}^n+a_{n-1}\tilde{D}^{n-1}+\cdots +a_1\tilde{D}+a_0$ (with all $a_i\in R$) satisfying
$\tilde{L}(x_i)=0$ for $i=1,\dots ,n$. A sufficient condition for the existence of $\tilde{L}$ is 
$\det ( \tilde{D}^jx_i)_{i=1,\dots ,n}^{j=0,\dots ,n-1}\in R^*$. 

If this condition is satisfied, then there are unique elements
$a_{n-1},\dots ,a_0$ for this relation. Then $a_{n-1},\dots ,a_0\in R^G=C(z)$. Then $\tilde{L}$ is the differential operator
associated to the evaluation $h$ of the invariants.\\

(b). {\it Suppose that $R$ has nilpotent elements}. Then one may replace $I$ by its radical $\sqrt{I}$ and $R$ by 
$R_{red}=R/\sqrt{I}$ and  follow the ideas of (a) above.  \\

 {\it Now we suppose that $R$ is reduced}. This is equivalent to $R$ is a finite \'etale extension of $C(z)$. The criterion for this is that the unit ideal of $C(z)[X_1,\dots ,X_n]$ is generated by $I$ and the determinants 
$\det \left(\frac{\partial f_j-h(f_j)}{\partial X_i}\right)_{i=1,\dots ,n}^{j\in J}$ where $J$ ranges over the subsets of $\{1,\dots ,N\}$ with $\# J=n$.

Since the elements $h(f_j)$ belong to $C(z)$ we may omit these in the determinants. 
Write $DET$ for $\det \left(\frac{\partial f_j}{\partial X_i}\right)_{i=1,\dots ,n}^{j=1,\dots ,n}$.\\
Then $df_1\wedge \cdots \wedge df_n=DET\cdot dX_1\wedge \cdots \wedge dX_n$. Thus for $\sigma \in G\subset {\rm GL}_n(C)$ one has $\sigma (DET)=\det (\sigma)^{-1} \cdot DET$. Since $G$ is finite, there exists an integer $m\geq 1$ with
$DET^m\in C[X_1,\dots ,X_n]^G$ and $h(DET^m)\in C(z)$.\\

{\it From $R$ reduced, it follows that (after renumbering)  $h(DET^m)\neq 0$} \\

The extension $\tilde{D}$ of $\frac{d}{dz}$ on $R$ lifts to a derivation $D$ on $C(z)[X_1,\dots ,X_n]$ with $D(z)=1$ and such that 
$D(I)\subset I$. The lift $D$ is not unique since one can add to a $D(X_i)$ any element in the ideal $I$.\\

The condition $D(I)\subset I$ with $I=(f_1-h(f_1),\dots ,f_N-h(f_N))$ is explicitly the following
\[  \sum _{j=1}^n\frac{\partial f_i}{\partial X_j}\cdot D(X_j)\equiv h(f_i)' \mod I, \mbox{ for  } i=1,\dots ,N.\] 

The assumption $h(DET^m)\neq 0$ is sufficient  for the definition of the vector $(DX_1,\dots ,DX_n)^t$ by the equation 
 \[ \left(\frac{\partial f_i}{\partial X_j}\right)(DX_1,\dots ,DX_n)^t=(h(f_1)',\dots ,h(f_n)')^t.\]  
  It follows that $D(f_i-h(f_i))\in I$ for all $i=1,\dots ,N$. Then one can derive formulas for $D^i,\ i=0,\dots ,n$. From this one deduces a linear combination  $L:=d_z^n+a_{n-1}d_z^{n-1}+\cdots +a_1d_z+a_0$ such that $L(x_i)=0$ for all
$i$. If this relation is unique then the $G$-invariant coefficients $a_j$ belong to $R^G=C(z)$. \\
{\it $L$ is the differential operator associated to the evaluation $h$.}  \hfill $\square$\\
}\end{algorithm}

\noindent 
{\it Observation}. (1). A successful application of the procedure depends heavily on properties of  $h$ (the evaluation of the invariants).
If $h$ is known to be the evaluation of an operator $L$, then the procedure produces $L$ up to (projective) equivalence.
For some choices of $h$ the operator $L$ does not exist. It can also happen that $L$ exists but has a differential Galois group 
which is a proper subgroup of $G$. \\
(2) Suppose that the evaluation of the invariants $h$ is changed into $h_\lambda$ given by $h_\lambda(f_i)= \lambda^{\deg f_i}h(f_i)$ for all $i$ and fixed $\lambda$ such that a power $\lambda ^m \in C(z)^*$ for some integer $m\geq 1$.
Then the new operator has the form $\lambda L\lambda ^{-1}$. Thus if $L=d_z^n+a_{n-1}d_z^{n-1}+\cdots +a_1d_z+a_0$, then 
the new operator is obtained from $L$ by the shift $d_z\mapsto d_z-\frac{\lambda'}{\lambda}$ (note that $\frac{\lambda'}{\lambda}\in C(z)$).
\begin{examples} {\rm In general, $I$ is not a maximal ideal of $C(z)[X_1,\dots ,X_n]$ and therefore  $R$ is not be field.
We consider, as in Remarks 3.10,  $G=A_4^{SL_2}$ and $C[x,y]^G=C[Q_3,Q_4,Q_6]$ with the relation
$Q_6^2-Q_3^4-4Q_4^3=0$. For the evaluations $h_1:(Q_3,Q_4,Q_6)\mapsto (z,0,z^2)$ and 
$h_2:(Q_3,Q_4,Q_6)\mapsto (0,z^2,2z^3)$ the above ideal $I$ is not maximal. In both cases, $R$ is a product of a number of copies of the field $C(z)$.\\

Procedure 4.3 applied to the evaluation $h_1$ leads to the {\it first} order differential operator $d_z-\frac{1}{6z}$ instead of a second
order operator. This is in accordance with the observation that for suitable $x_0,y_0\in C^*, x_0\neq y_0$ one has 
$Q_3(x_0z^{1/6},y_0z^{1/6})=z,\ Q_4(x_0z^{1/6},y_0z^{1/6})=0,\ Q_6(x_0z^{1/6},y_0z^{1/6})=z^2$. 
Moreover the differential Galois group is $C_6$, the cyclic group of order 6, which can be seen as a subgroup of $A_4^{SL_2}$.\\

The procedure does not produce an operator for $h_2$. Indeed, $Q_3$ is a product of six linear forms in the two $C$-linearly 
independent solutions $x,y$ and $h_2(Q_3)=0$ contradicts this linear independence.  \hfill $\square$\\
}\end{examples}

\begin{examples} {\rm  Consider the group $S_5$, acting on the vector space $W:={C}e_1\oplus \cdots \oplus {C}e_5/{C}(e_1+\dots +e_5)$. The algebra $A:={C}[x_1,\dots ,x_5]/(x_1+\dots +x_5)$ is given the induced action. 

The invariants are the elementary functions $E_2,E_3,E_4,E_5$, homogeneous polynomials
of degrees 2,3,4,5.  There are no relations. For the action of $S_5$ on $\mathbb{P}(W)\cong \mathbb{P}^3$ we are looking for an irreducible invariant curve $Z$ such that $Z/S_5$ has genus 0. 
  Moreover we want to find an evaluation. \\
  
 Let $Z$ be defined by $E_2=E_3=0$. Thus $Z$ as subset of $\mathbb{P}^4$ is given by $E_1=E_2=E_3=0$. The homogeneous coordinate ring of $Z$ is $B:={C}[x_1,\dots ,x_5]/(E_1,E_2,E_3)$. It can be seen that $B$ is a domain. 
  
  Now $C:=B^{S_5}={C}[E_4,E_5]$ and $C_{((0))}$ is generated by  $\frac{E_4^5}{E_5^4}$. We send this element to $z$. This induces an evaluation $ev: (E_2,E_3,E_4,E_5)\rightarrow (0,0,z,z)$.         
  
   Now $A/I\rightarrow K\supset {C}(z)$, where $I$ is the ideal generated by $\{E_2,E_3,E_4-z,E_5-z\}$ and $K$ denotes the Picard--Vessiot field of the, to be computed, differential equation.

 Associated to the evaluation is the polynomial $X^5+zX-z$, namely the evaluation of $X^5+E_2X^3-E_3X^2+E_4X-E_5$. One sees that $K$ is the splitting field of the polynomial  $X^5+zX-z$ over ${C}(z)$. Further, a basis of solutions  consists of 4 of the 5 zero's of
 $X^5+zX-z$ (as algebraic functions of $z$).    The computation using the Procedure 4.3 yields $L$ equals
 \begin{small}
 \[d_z^4+\frac{4(416z+3125)}{z(256z+3125)}d_z^3+\frac{60(36z+125)}{z^2(256z+3125)}d_z^2+\frac{360}{z^2(256z+3125)}d_z+\frac{120}{z^4(256z+3125)}.\] \end{small}
 } \end{examples}

\begin{lemma} Let $A$ be a finitely generated graded $C$-algebra. Assume that $A$ is a domain and that $A_{((0))}=C(z)$.
Then there exists a $C$-algebra homomorphism $h:A\rightarrow C[z]$ such that $h$ induces the identification $A_{((0))}=C(z)$. 
\end{lemma}
\begin{proof} Write $A=C[f_1,\dots ,f_r]$ where the $f_1,\dots ,f_r$ are homogeneous elements of degrees $d_1,\dots ,d_r \in \mathbb{Z}_{>0}$.
Let $v(i)=(v(i)_1,\dots ,v(i)_r)$  for $i=1,\dots ,r-1$ denote free generators of  $\{(n_1,\dots ,n_r)\in \mathbb{Z}^r\ |\ \sum n_id_i=0\}$. We may and will suppose that the matrix $\{ v(i)_j\}_{i,j=1}^{r-1}$ is invertible. Let $m\in \mathbb{Z}_{\neq 0}$ be its determinant.\\

The elements $\{ f_1^{v(i)_1}\cdots f_r^{v(i)_r}\ |\ i=1,\dots ,r-1\}$ generate the field $A_{((0))}=C(z)$ over $C$ and thus we can identify
$ f_1^{v(i)_1}\cdots f_r^{v(i)_r}$ with some $\alpha _i\in C(z)$. First we define 
$\tilde{h}:A\rightarrow \overline{C(z)}$ by
$\tilde{h}(f_r)=1$ and the $\tilde{h}(f_1),\dots ,\tilde{h}(f_{r-1})$ are such that 
$\tilde{h}(f_1)^{v(i)_1}\cdots \tilde{h}(f_{r-1})^{v(i)_{r-1}}=\alpha _i$ for $i=1,\dots ,r-1$. 
One observes that the $\tilde{h}(f_i)$ are Laurent polynomials in $\alpha _1^{1/m},\dots ,\alpha _{r-1}^{1/m}$. Thus the expressions $\tilde{h}(f_i)$
have the form $R(z)\cdot (z-a_1)^{n_1/m}\cdots (z-a_s)^{n_s/m}$ with $R\in C(z)$, certain distinct $a_1,\dots ,a_s\in C$
and certain integers $n_i\in \{0,\dots ,m-1\}$.\\

 The algebraic relations between the
$f_1,\dots ,f_r$ are generated by homogeneous relations. Hence for any expression $\lambda \in \overline{C(z)}^*$ we can consider
the $C$-algebra homomorphism $h$ is given as $h(f_i)=\lambda ^{d_i} \tilde{h}(f_i)$ for $i=1,\dots ,r$. It is seen that for suitable $\lambda$
the $m$th roots and the denominators disappear. Thus the required $h$ exists and can be seen to be unique (up to constants) under the condition 
that $\sum _{i=1}^r\deg h(f_i)$ is minimal.    \end{proof}

The evaluations $h$ and $\tilde{h}$ are said to be {\it essentially the same}.

\begin{corollary} Let be given an irreducible  finite group $G\subset {\rm SL}(V)$ and an irreducible $G$-invariant curve $Z\subset \mathbb{P}(V)$
such that the function field of $Z/G$ is $C(z)$. Lemma {\rm 4.5} produces an evaluation of the invariants $h:C[V]^G\rightarrow C[z]$ 
which induces the identification of the function field of $Z/G$ with $C(z)$. 

This evaluation $h$ is essentially the same as
the evaluation of the invariants induced by the standard operator $L_{st}$ for the data $G$ and $Z$ (see \S 3).
\end{corollary}
\begin{proof} Let $H\subset C[V]$ be the (prime) homogeneous ideal of $Z\subset \mathbb{P}(V)$.  Then $H\cap C[V]^G$ defines
the curve $Z/G$ and the homogeneous algebra of $Z/G$ is $A:=C[V]^G/(H\cap C[V]^G)$. 
Now one applies Lemma 4.5 to $A$. The last statement follows from the unicity of $h$ up to  a change 
$h(f_i)\mapsto \lambda ^{\deg f_i} h(f_i)$ for $i=1,\dots ,r$.   \end{proof}

\begin{remarks} {\rm Suppose that the finite group $G\subset {\rm SL}(V)$ and an explicit presentation by homogeneous generators
$f_1,\dots ,f_r$ and relations for $C[V]^G$ are given. One can try to find $C$-algebra homomorphisms $h:C[V]^G\rightarrow C[z]$ such that
$C(z)$ is generated over $C$ by $h(f_1),\dots ,h(f_r)$. Moreover one may suppose that $h$ is minimal in the sense that
a change $h(f_i)\mapsto \lambda ^{\deg f_i}h(f_i)$ for $i=1,\dots ,r$ does not produce an evaluation $C[V]^G\rightarrow C[z]$ with lower
degrees.

 The kernel of $h$ is a homogeneous prime ideal $\mathcal{P}$ (of codimension 1) in $C[V]^G$.  
Let $\mathcal{Q}\subset C[V]$ denote a homogeneous prime ideal above $\mathcal{P}$. If $\mathcal{Q}$ is unique, then
its stabilizer is $G$ and it defines a $G$-invariant curve $Z\subset \mathbb{P}(V)$ with $Z/G$ of genus zero and evaluation essentially 
equal to $h$. Then Procedure 4.3 produces an operator $L$ with differential Galois group $G$.

 In general, $\mathcal{Q}$ is not unique and its stabilizer is a proper subgroup $G'$ of $G$. Then Procedure 4.3 will produce a differential 
 operator with differential Galois group $G'$. See Examples 4.4 and  \S 5.4   }\end{remarks}

\section{\rm Computations with Procedure 4.3}\label{five}

\subsection{\rm Finite $G\subset {\rm SL}(V)$ with $\dim V=2$.}
For the finite subgroups of ${\rm SL}_2$ and their invariants we use  the notations and equations from \cite{Ca}. \\
 
\noindent (1). $D_n^{SL_2}$ is generated by 
${\zeta \ 0 \choose 0 \  \zeta ^{-1} }, {0\ -1\choose 1\ 0 }$ with $\zeta =e^{2\pi i/2n}$. The group has $4n$ elements; 
the semi-invariants are generated by $f_3=xy$, $f_{12}=x^{2n}+y^{2n}$, $f_{13}=x^{2n}-y^{2n}$ and the invariants have generators\\
$F_1=f_3f_{13}, F_2=f_{12}, F_3=f_3^2$ and relation $F_1^2-F_2^2F_3+4F_3^{n+1}=0$. \\
One computes two generators $v_1,v_2$ of $\{ (a_1,a_2,a_2)\in \mathbb{Z}^3  \ | \ a_1\deg F_1+a_2\deg F_2+a_3 \deg F_3=0\}$. 
The expressions of degree zero $V_1,V_2$ corresponding to $v_1,v_2$ are generators of the field
$A_{((0))}$ over ${C}$. This leads to a suitable choice of an identification
$A_{((0))}={C}(z)$.  Now one can define the evaluation $h$ by assuming $h (F_3)=1$. This might lead to
roots which can be avoided by multiplying $F_1,F_2,F_3$ by $w^{2n+2},w^{2n},w^4$ for a suitable $w$.\\

\noindent $n=2$. The generators $(0,-1,1),(2,-3,0)$ are mapped to $\frac{F_3}{F_2}, \frac{F_1^2}{F_2^3}$.
Further $\frac{F_1^2}{F_2^3}=  \frac{F_3}{F_2}-4$. A choice for $z$ by $\frac{F_3}{F_2}=\frac{1}{2z}$. The evaluation of the invariants is\\
$(F_1,F_2,F_3)\mapsto (2(z^2-1)^{1/2},2z,1)$ or $\mapsto (2(z^2-1)^2,2z(z^2-1),(z^2-1))$.\\

\noindent $n$ is odd. Generators $( 1 ,0 ,-(n+1)/2 ),(0,2,-n)$ and expressions 
$\frac{F_1}{F_3^{(n+1)/2}},\frac{F_2}{F_3^n}$. Equation $(\frac{F_1}{F_3^{(n+1)/2}  })^2=\frac{F_2}{F_3^n}-4$.
A choice for $z$ is  $\frac{F_1}{F_3^{(n+1)/2}  }=2iz$. Then the evaluation of the invariants is
$(F_1,F_2,F_3)\mapsto (2iz,2i(z^2-1)^{1/2},1)$ or $\mapsto (2iz(z^2-1)^{(n+1)/2},2i(z^2-1)^{(n+1)/2},(z^2-1))$.\\

\noindent $n>2$ is even. Generators $( 2,0 ,-(n+1) ), (0 ,1 ,-n/2 )$, mapping to 
$\frac{F_1^2}{F_3^{n+1}},\frac{F_2}{F_3^{n/2}}$. Relation 
$\frac{F_1^2}{F_3^{n+1}}=(\frac{F_2}{F_3^{n/2}})^2-4$. Choice for $z$ by
 $\frac{F_2}{F_3^{n/2}}=2z$. Evaluation $(F_1,F_2,F_3)\mapsto (2(z^2-1)^{1/2},2z,1)$ or $(2(z^2-1)^{1+n/2}, 2z(z^2-1)^{n/2}, (z^2-1))$. 
  {\it For all cases the differential operator is
$L=d_z^2+\frac{z}{z^2-1}d_z-\frac{1}{4n^2(z^2-1)}$.} This becomes the standard equation of Examples 3.9 after $z\mapsto 2z -1$.

 \bigskip

   \noindent 
  (2). $A_4^{SL_2}$. This is a continuation of Remarks 3.10. Generators for the invariants are the homogeneous polynomials
  $Q_3,Q_4,Q_6$ of degrees $6,8,12$ and with relation $Q_6^2=Q_3^4+4Q_4^3$. Further 
  $A_{((0))}=C(\frac{Q_6}{Q_3^2},\frac{Q_4^3}{Q_3^4})=C(z)$ with $z=\frac{Q_6}{Q_3^2}$. One starts with the evaluation
  $(Q_3,Q_4,Q_6)\mapsto (1,(\frac{z^2-1}{4})^{1/3},z)$ which is transformed into the ``minimal'' evaluation
  $((\frac{z^2-1}{4})^2,(\frac{z^2-1}{4})^3,(\frac{z^2-1}{4})^4z)$.\\
 {\it  The differential operator is $d_z^2+\frac{27z^2+101}{144(z^2-1)^2}$}. This becomes the standard
  equation after $z\mapsto 2z -1$.
\bigskip

\noindent 
(3). $S_4^{SL_2}$ has ring of invariants $A:={C}[F_1,F_2,F_3]$ and the degrees are $12,8,18$.
Generators for the relations are $(2,-3,0),(3,0,-2)$. The corresponding functions are 
$\frac{F_2^3}{F_1^2}, \frac{F_3^2}{F_1^3}$ and have the relation 
$\frac{F_2^3}{F_1^2}=\frac{F_3^2}{F_1^3}+108$. A possible evaluation is 
$(F_1,F_2,F_3)\mapsto (1,3\cdot 2^{2/3}z,2\cdot 3^{3/2}(z^3-1)^{1/2})$. This can be improved into
$(2^23^3(z^3-1)^3,2^2 3^3 z(z^3-1)^2, 2^4 3^6 (z^3-1)^5 )$.\\
 {\it The differential operator is $d_z^2+\frac{(7z^3+101)z}{64(z^2+z+1)^2(z-1)^2}$. }
 The equation has 4 singular points and is a pullback of the standard equation. One sees that the given 
 evaluation is not `minimal', i.e., the field $A_{((0))} = C(s)$ with $s = 1-z^3$. The equation is standard in the variable $s$.

\bigskip

\noindent 
(4). The group $A_5^{SL_2}$ has ring of invariants $A:={C}[f_9,f_{10},f_{11}]$ and the degrees are 
$30,20,12$ and relation $f_9^2+f_{10}^3-1728f_{11}^5=0$.
 Generators for the field $A_{((0))}$ are $\frac{f_9^2}{f_{11}^5},\frac{f_{10}^3}{f_{11}^5}$ and there is the relation
 $\frac{f_9^2}{f_{11}^5}=-\frac{f_{10}^3}{f_{11}^5}+1728$. This leads to the evaluation
$(f_9,f_{10},f_{11})\mapsto (-1728(z-1))^{1/2}, (1728z)^{1/3},1)$. One can change this into 
$( (-1728)^{1/2}z^{10}(z-1)^8,(1728)^{1/3}z^7(z-1)^5,z^4(z-1)^3)$. Also the constants can be improved in a similar way
to 
\[( (1728)^{3}z^{10}(z-1)^8,-(1728)^{2}z^7(z-1)^5,-(1728)z^4(z-1)^3).\]
{\it The differential operator is $d_z^2+\frac{864z^2-989z+800}{3600z^2(z-1)^2}$.}
This is the standard equation after $z \mapsto 1-z$. 
 
\subsection{\rm  $G=G_{168}\subset {\rm SL}(V)$ and $\dim V=3$.}
\subsubsection{\rm Computation of the differential equation related to Klein's quartic}
  We use the formulas of \cite{Ber}, p 50. 
 ${C}[X_1,X_2,X_3]^G={C}[F_4,F_6,F_{14},F_{21}]/(rel)$, where
 $F_4,F_6,F_{14},F_{21}$ are homogeneous polynomials of degrees $4,6,14,21$ and there is one relation. 
 The Klein quartic $Z\subset \mathbb{P}(V)$ is given by $F_4=0$ with $F_{4}:=2(X_1X_2^3+X_2X_3^3+X_3X_1^3)$.\\

Unlike the case $\dim V=2$, a direct computation of the standard operator for these data with the methods of \S 3 meets difficulties.
${C}(Z)$ is the field of fractions of ${C}[ \frac{X_2}{X_1},\frac{X_3}{X_1}]/(X_1^{-4}F_4)$. How to compute
$f\in C(Z)^*$ such that $\frac{\sigma (f)}{f}=\frac{\sigma X_1}{X_1}$ for all $\sigma \in G$?
How to compute the derivatives w.r.t. $d_t=\frac{d}{dt}$ of a basis of  the solution space $W=<f,fX_2/X_1,fX_3/X_1>$?\\

We continue with the methods of   \S 4.3. The graded algebra of $Z/G$ is 
\[A:=\left\{{C}[X_1,X_2,X_3]/(F_4)\right\}^G={C}[F_6,F_{14},F_{21}]/(F_{21}^2-4F_{14}^3-54F_6^7).\]
The field $A_{((0))}=C(Z/G)$ is generated over $C$ by $\frac{F_{21}^6}{F_6^{21}}$ and $\frac{F_{14}^3}{F_6^7}$. The relation 
 reads $\frac{F_{21}^6}{F_6^{21}}=(4\frac{F_{14}^3}{F_6^7}+54)^3$ and so $A_{((0))}=C(t)$ with
  $t=\frac{F_{14}^3}{F_6^{7}}$. 
The evaluation of the invariants $h:C[X_1,X_2,X_3]^G\rightarrow C(t)$ is given by
 \[h:(F_4,F_6,F_{14},F_{21})\mapsto  (0,t^2(4t+54)^3,t^5(4t+54)^7,t^7(4t+54)^{11}).\] 
The Procedure 4.3 produces an operator  $S_0$ with singularities $t=0,-\frac{27}{2},\infty$ and the local exponents 
are $1,2/3,1/3||1,1/2,3/2||-3/7,-5/7,-6/7$. \\

The  change $t=\lambda z$ and $d_t=\lambda^{-1}d_z$ and $\lambda =-\frac{27}{2}$ yields singularities at
$0,1,\infty$ with the same local exponents. This leads to the operator $S_1=$ 
\begin{small}
\[  d_z^3+\frac{1}{z}d_z^2+\frac{72z^2+61z+56}{252z^2(z-1)^2}d_z-
\frac{6480z^3+3945z^2+13585z-5488}{24696z^3(z-1)^3.}\] 
\end{small}

The operator $ S_2:=z^{-1}(z-1)^{-1}S_1z(z-1)$ has the ``classical'' local exponents and coincides with the formulas in 
the literature \cite{H,S-U,vdP-U}.
\[S_2=d_z^3+\frac{7z-4}{z(z-1)}d_z^2+\frac{2592z^2-2963z+560}{252z^2(z-1)^2}d_z+
 \frac{\frac{72\cdot 11}{7^3}z-\frac{40805}{24696}}{z^2(z-1)^2}.  \]

\subsubsection{\rm The Hessian of the Klein quartic}
 This is the $G_{168}$-invariant curve  $Z\subset \mathbb{P}(V)$ with equation $F_6=0$. 
 The graded algebra of $Z/G_{168}$ is $C[F_4,F_{14},F_{21}]/(F_{21}^2-4F_{14}^3+8F_{14}F_4^7)$.
 The field ${C}(Z)^{G_{168}}$ is ${C}(t)$ with $t=\frac{F_{14}^2}{F_4^7}$.   This produces 
\[h : (F_4,F_6,F_{14},F_{21})\rightarrow (t^3(t-2)^2, 0,  t^{11}(t-2)^7,2t^{16}(t-2)^{11}).\]
Procedure 4.3 produces the following differential operator (after a change of variables)
\[ d_z^3+\frac{3(3z-2)}{2z(z-1)}d_z^2+\frac{3(116z-35)}{112z^2(z-1)}d_z+\frac{195}{2744z^2(z-1)}.\]

\subsubsection{\rm More third order operators with group $G_{168}$}
The third order operators over $C(z)$, or more precisely, the differential modules of dimension 3, with singular points $0,1,\infty$
and differential Galois group $G_{168}$ are classified in \cite{vdP-U}, using the ``transcendental'' Riemann--Hilbert correspondence. Each case is given by a branch type
$[e_0,e_1,e_\infty]$ and a choice of one of the two irreducible characters $\chi_2,\chi_3$ of dimension 3. 
 The {LIST} is:\\ {\it
$[2,3,7]$, 1 case $g=3$; $[2,4,7]$, 1 case, $g=10$; $[2,7,7]$, 1 case, $g=19$; \\
$[3,3,4]^*$, 2 cases, $g=8$;  $[3,3,7]$, 1 case, $g=17$; $[3,4,4]$, 1 case, $g=15$; $[3,4,7]^*$, 2 cases, $g=24$; 
$[3,7,7]$, 2 cases, $g=33$; $[4,4,4]^*$, 2 cases, $g=22$; $[4,4,7]$, 1 case, $g=31$; $[4,7,7]^*$, 2 cases, $g=40$;  $[7,7,7]$, 1 case, $g=49$. \\}
For many cases in LIST these data lead to a computation of the third order operator. The cases where this fails are indicated by a $*$. \\

If one can identify for an item of LIST the $G_{168}$-invariant (Fano) curve $Z\subset \mathbb{P}(V)\cong \mathbb{P}^2$, then one obtains an evaluation and Procedure 4.3 produces this third order differential operator.  The items $[2,3,7],[2,4,7]\in$ LIST corresponds to the curves $F_4=0$ and $F_6=0$. In \cite{Ber} the smooth $G_{168}$-invariant $Z\subset \mathbb{P}(V)$, such that the normalisation of $Z/G_{168}$ has genus 0, are classified.  They correspond to equations $F_4, F_6,F_{14},\lambda F_6^3+F_{14}F_4, \lambda F_4^5+F_{14}F_6, \mbox{ with } \lambda \in C^*$.  Only the first two have three branch points.\\

{\it We extend this search by studying the invariant curves  $aF_4^3+F_6^2=0$}.\\  Consider the evaluation 
$h:(F_4,F_6,F_{14})\mapsto (1,\lambda ,t)$ with  $\lambda \in {C}^*$, $a=-\lambda ^2$. Then
$h(F_{21})^2$ is $(4t^3-44\lambda t^2+(126\lambda^4+68\lambda^2-8)t+54\lambda^7-938\lambda^5+
 172\lambda^3-8\lambda)$. The discriminant of this polynomial in $t$ is $-64(27\lambda^2-2)^3(\lambda^2+2)^4$. 
 The values $\lambda =(-2)^{1/2}$ and $\lambda =(2/27)^{1/2}$ are special.  In both cases there at most 4 branch points. For other values of $\lambda$ there are more branch points.\\
 The evaluations for $\lambda =(-2)^{1/2}$ and $\lambda =(2/27)^{1/2}$ are
  \[F_4,\ F_6, \ F_{14},\ F_{21}\mapsto w^4, \ w^6\sqrt{-2},\ w^{14}t, w^{21}2(-t+9\sqrt{-2})(t+7\sqrt{-2})^{1/2},\mbox{ and }\]
 \[F_4,\ F_6, \ F_{14},\ F_{21}\mapsto  w^4, \ w^6\sqrt{2/27},\ w^{14}t, 
  w^{21}\frac{-2\sqrt{3}}{243}(27t+\sqrt{6})(-27t+35\sqrt{6})^{1/2}.\] 
 Procedure 4.3 will produce the corresponding operators. We note that the above ``Fricke pencil of curves'' has been studied
 by M. Kato (see \cite{Ka}, Proposition 2.3) using Schwarz maps. The two special cases above  were found and the corresponding third order differential operators are computed. The operators have three branch points and  the solutions are
 in fact  hypergeometric functions.\\

\noindent  {\it Computing the evaluation for operators $L$ in LIST.}\\

An element in LIST is given by a topological covering of $\mathbb{P}^1\setminus \{0,1,\infty \}$ with group $G_{168}$,
produced by a triple $g_0,g_1,g_\infty\in G_{168}$ satisfying  $g_0g_1g_\infty=1$ and generating $G_{168}$. 
One may hope that from a given triple one can read off a part of the evaluation $h$ of the operator, namely 
the orders of the functions $h(Q_4),h(Q_6),h(Q_{14}),h(Q_{21})$ at the points $0,1,\infty$.  

It can be seen that the knowledge of these orders and the relation between the four invariants is sufficient for the computation
of suitable $h$. We illustrate this in more detail for the item $[2,4,7]\in$ LIST. A search for invariants of degrees $4,6,14,21$ leads to a formula

\begin{footnotesize}
\[\left( h(Q_4),h(Q_6),h(Q_{14}),h(Q_{21})\right)= 
\left(\frac{f_4}{t^2(t-1)^3},\frac{f_6}{t^3(t-1)^4},\frac{f_{14}+g_{14}t}{t^7(t-1)^{10}},\frac{f_{21}(t+2400)}{t^{10}(t-1)^{15}}\right),\] 
\end{footnotesize}
where $f_{4},f_{6},f_{14},g_{14},f_{21}$ are constants. The equation satisfied by the invariants produces an actual solution 
$(f_{4},f_{6},f_{14},g_{14},f_{21})=(\frac{-7}{4},\frac{-3}{4},\frac{-149}{8},\frac{1}{4},\frac{1}{8})$. \\

\noindent {\it Evaluation of items in LIST for $G_{168}$}.\\
The first row gives the branch type, the local exponents at $0,1,\infty$ and the value of $\mu$.. The second row gives the evaluation for the invariants
$F_4,F_6,F_{14},F_{21}$.
\begin{footnotesize}
\begin{itemize}
\item{[2,3,7]} $ -\dfrac{1}{2},0,\dfrac{1}{2} ||  -\dfrac{1}{3},-\dfrac{1}{3},0||\dfrac{8}{7},\dfrac{9}{7},\dfrac{11}{7}||\dfrac{12293}{24696}$\\
  $0$;$\dfrac{-3^3}{t^{3}(t-1)^{4}}$;$\dfrac{2^2 3^8}{t^{7}(t-1)^{9}}$;$\dfrac{2^3 3^{12}}{t^{10}(t-1)^{14}}$.
\item{[2,4,7]} $ -\dfrac{1}{2},0,\dfrac{1}{2}||-\dfrac{3}{4},-\dfrac{1}{4},0||\dfrac{8}{7},\dfrac{9}{7},\dfrac{11}{7}||\dfrac{5273}{10976}$\\
$\dfrac{-7}{2t^2(t-1)^2}$;$-\dfrac{3}{4t^{3}(t-1)^{4}}$;$-\dfrac{(-149+2t)}{8t^{7}(t-1)^{10}}$;$\dfrac{(t+2400)}{8t^{10}(t-1)^{14}}$.
\item{[2,7,7]} $ -\dfrac{1}{2},1,\dfrac{1}{2}||-\dfrac{6}{7},-\dfrac{5}{7},-\dfrac{3}{7}||\dfrac{8}{7},\dfrac{9}{7},\dfrac{11}{7}||\dfrac{1045}{686}$\\
 $\dfrac{14}{t^2(t-1)^3}$;$\dfrac{3}{t^{2}(t-1)^{5}}$;$\dfrac{4(-294+294t+t^2)}{t^{6}(t-1)^{12}}$;$\dfrac{8(t-2)(t^2-9604t+9604)}{t^{9}(t-1)^{18}}$.
\item{[3,3,7]} $-\dfrac{2}{3},-\dfrac{1}{3},0||-\dfrac{2}{3},-\dfrac{1}{3},0||\dfrac{9}{7},\dfrac{11}{7},\dfrac{15}{7}||0$ \\
 $0$;$-\dfrac{2^4 3^3}{t^4(t-1)^4}$;$\dfrac{2^{12}3^8}{t^{9}(t-1)^{9}}$;$\left(1-2t\right)\dfrac{2^{17}3^{12}}{t^{14}(t-1)^{14}}$.
\item{[3,7,7]} $ -\dfrac{2}{3},-\dfrac{1}{3},0||-\dfrac{6}{7},-\dfrac{5}{7},-\dfrac{3}{7}||\dfrac{10}{7},\dfrac{13}{7},\dfrac{19}{7}||\dfrac{830}{1029}$\\
  $0$;$\dfrac{3^3}{t^4(t-1)^5}$;$\dfrac{3^8(9t-8)}{t^{9}(t-1)^{12}}$;
  $\dfrac{3^{12}(27t^2-36t+8)}{t^{14}(t-1)^{18}}$.
\item{[4,4,7]} $-\dfrac{3}{4},-\dfrac{1}{4},0||-\dfrac{3}{4},-\dfrac{3}{4},0||\dfrac{9}{7},\dfrac{11}{7},\dfrac{15}{7}||0$ \\
$\dfrac{-14}{t^3(t-1)^3}$;$\dfrac{-12}{t^4(t-1)^4}$;$\dfrac{256t^2-256t-4704}{t^{10}(t-1)^{10}}$;
$\dfrac{512(2t-1)(4t^2-4t+2401)}{t^{15}(t-1)^{15}}$.
\item{[7,7,7]} $-\dfrac{6}{7},-\dfrac{5}{7},-\dfrac{3}{7}||-\dfrac{6}{7},-\dfrac{5}{7},-\dfrac{3}{7}||\dfrac{9}{7},\dfrac{11}{7},\dfrac{29}{7}||0$\\
 $\dfrac{16}{t^3(t-1)^3}$;$\dfrac{512t^2-512t+5}{16t^5(t-1)^5}$;$\dfrac{P_6(t)}{t^{12}(t-1)^{12}}$;
 $\dfrac{(2t-1)P_8(t)}{2^{12}t^{18}(t-1)^{18}}$.
\end{itemize}
where $P_6(t)=2^{12}t^6-12288t^5+49280t^4-78080t^3+\dfrac{74441}{2}t^2-\dfrac{457}{2}t+\dfrac{1}{2^8}$ and\\ 
$P_8(t)=536870912t^8-2147483648t^7-74398564352t^6+230711885824t^5\\
-231821246464t^4+76617285632t^3+502637824t^2-1385728t-1$\\
\end{footnotesize}

One observes from these data a relation between the local exponents and the orders of the invariants at the points
$0,1,\infty$. Another observation is that items $[3,3,7]$ and $[3,7,7]$ are weak pullbacks of the ``standard'' equation
$[2,3,7]$.   The pullback functions are respectively $\phi (t)=4t(t+1)+1$ and $\phi(t)=-\frac{(27t^2-36t+8)^2}{t-1}$.

The Fano curve for $[2,4,7],\ [2,7,7],\ [4,4,7]$ is $-\frac{7}{54}F_6^3-\frac{1}{8}F_4F_{14}+F_4^3F_6$ and has genus 10.
$[2,7,7]$ and $[4,4,7]$ are weak pullbacks of the ``standard'' equation $[2,4,7]$ with pullback functions $\phi(t)=\frac{-(t-1)^2}{4(t-1)}$ and $\phi(t)=(2t-1)^2$. The Fano curve for $[7,7,7]$ is a long expression of degree 36.

  \subsection{\rm $G=H_{72}\subset {\rm SL}(V)$ and $\dim V=3$.} 
  The group $G=H_{72}$ and its invariants are described as group $F$ in the  \cite{Ca}, p. 59
  and occurs as the second example,Th\'eor\`eme 4, in   \cite{Ro}.  We adopt the last description:
  \begin{footnotesize}
  \[P=xyz;  Q=x^3y^3+x^3z^3+y^3z^3;  S=x^3+y^3+z^3\]
  \[F_1=S^2-12Q, F_2=(x^3-y^3)(x^3-z^3)(y^3-z^3),F_3=S^4+216P^3S,F_4=(S^2-18P^2-6PS)^2.\]
  \end{footnotesize}
  The algebra of invariants is $C[x,y,z]^G=C[F_1,F_2,F_3,F_4]$ and there is one relation
  $(432F_2^2+3F_1F_3-F_1^3)^2-4(F_4^3-3F_4^2F_{3}+3F_4F_{3}^2)=0$.

  The choice of the $G$-invariant irreducible curve $Z\subset \mathbb{P}(V)$ is given by $F_1=0$. The graded algebra of 
  $Z/G$ is $A={C}[F_2,F_3,F_4]/(432^2F_2^4-4(F_4^3-3F_4^2F_3+3F_4F_3^2))$.
The field $A_{((0))}$ is generated by $\frac{F_2^4}{F_3^3}, \frac{F_4}{F_3}$ and there is the equation
 $432^2\frac{F_2^4}{F_3^3}= 4((\frac{F_4}{F_3})^3 -3 (\frac{F_4}{F_3})^2+3\frac{F_4}{F_3})$. Hence $A_{((0))}$ is generated 
 by $\frac{F_4}{F_3}$. This leads to the evaluation $F_1\mapsto 0$ and 
 $(F_2,F_3,F_4)\mapsto ( \sqrt[4]{   \frac{t^3-3t^2+3t}{6^6} }, 1 , t )$. This can be simplified as follows.
 Write $w= \frac{t^3-3t^2+3t}{6^6}$ and multiply the values for $F_2,F_3,F_4$ by $w^{3/4},w,w$.
 the result is  $(F_1,F_2,F_3,F_4)\mapsto (0,w, w, wt )$. \\
 Procedure 4.3 produces the differential operator
 \[d_t^3+\frac{5t^3-15t^2+15t-6}{(t^3-3t^2+3t)(t-1)}d_t^2+\frac{(160t^3-480t^2+480t-117)(t-1)}{48(t^3-3t^2+3t)^2}d_t\] 
\[-\frac{(160t^3-480t^2+480t-189)(t-1)^3}{432(t^3-3t^2+3t)^3}.\]
On observes that $t=1$ is an apparent singularity and that $\infty$ and the three roots of  $t^3-3t^2+3t$ are the 4
singular points.

{\it Remarks}.   At present we have no differential equation over $C(t)$ with three singularities.
A differential equation of order 3 with Galois group $H_{72}$ was also found by M. van Hoeij, see section 2 of
 \cite{Ho}.

 \subsection{\rm  Differential equations for $G=A_5$ and $\dim V=3$.}

$A_5$ is explicitly given as a subgroup of ${\rm SL}_3(\mathbb{Q}(\zeta_5))\subset {\rm SL}_3(\mathbb{C})$ where $\zeta_5=e^{2\pi i/5}$ by the generating matrices (copied from \cite{Ber}).
 This corresponds to the irreducible character (say) $\chi_2$ of dimension 3 for $A_5$. The other irreducible character $\chi_3$ of dimension 3 is obtained from the automorphism $\zeta_5\rightarrow \zeta_5^{2}$ of the field $\mathbb{Q}(\zeta_5)$. 
Further $A_5^{SL_2}\subset {\rm SL}_2(\mathbb{C})$ is the preimage of $A_5$ under the  ``second symmetric power map''
${\rm SL}_2\rightarrow {\rm SL}_3$ given by $A\mapsto sym^2A$. The group $A_5^{SL_2}$ has two irreducible characters of dimension 2 and their
second symmetric powers are the above 3-dimensional characters $\chi_2,\chi_3$ of $A_5$. The following proposition is probably well known. Due to  lack of reference we sketch a proof (related to theorem 2.1 in  \cite{N-vdP}).
 
 \begin{proposition} {\em Comparing differential modules for $A_5$ and $A_5^{SL_2}$}.\\  
 {\rm (1)}. Suppose that the 3-dimensional differential module $M$ over $C(z)$ has differential Galois group $A_5$.
 Then there is a 2-dimensional differential module $N$ with differential Galois group $A_5^{SL_2}$ such that 
 $sym^2N$ is isomorphic to $M$. \\
  \noindent {\rm (2)}. The module $N$ is unique up to tensoring with a 1-dimensional module $D$ such that $D^{\otimes 2}={\bf 1}$, where ${\bf 1}$ denotes the trivial differential  module. 
\end{proposition}
\begin{proof} The action of $A_5$ on the solution space $W$ of $M$ induces an action on $sym^2W$. It has an invariant line 
and  this corresponds to a  1-dimensional submodule $T$ of $sym^2(M)$. 
A non zero element of $T$  is a non degenerate quadratic form in terms of a basis of $M$.
This form has a non trivial zero over ${C}(t)$ since the latter is a $C_1$-field. Thus there is a basis $x_1,x_2,x_3$ of $M$ such that
$T$ is generated by $x_1x_3-x_2^2$. Moreover $T$ is the trivial module since $A_5$ is simple.  For some $q\in \mathbb{C}(t)^*$ one has
 $\partial (q(x_1x_3-x_2^2))=0$.\\   
  
 The equation $\frac{1}{q}\partial (q(x_1x_3-x_2^2))=0$ implies that the matrix $A$ of $\partial$ w.r.t. the basis $x_1,x_2,x_3$
of $M$ has the form $\left(\begin{array}{ccc} a_1& b_1&0 \\ 2b_3 &\frac{-q'}{2q} &2b_1 \\ 0 & b_3 & -a_1-\frac{q'}{q} \end{array}\right)$.
 Since $A_5$ is simple, $\det M={\bf 1}$. This implies that the equation $y'=tr(A)y$  has a solution in ${C}(t)$. Therefore
  $q^{-3/2}\in {C}(t)$ and thus $q$ is a square.  \\

  We may suppose $q=1$. Now $\partial (x_1x_3-x_2^2)=0$ implies that the matrix of $\partial$ with respect to the basis
  $\{x_1,x_2,x_3\}$  has the  form  $\left(\begin{array}{ccc}  2a&b &0 \\  2c& 0& 2b \\  0&c &-2a \end{array}\right)$ for certain $a,b,c$. 
  Consider the 2-dimensional  module $N$ and a basis $y_1,y_2$  such that the matrix of $\partial$ is ${a\ b\choose c\ -a }$.
  Then $sym^2(N)$ has on basis $x_1=y_1^2,\ x_2=y_1y_2,\ x_3=y_2^2$ the above matrix. Thus $sym^2N\cong M$. 
  The differential Galois group $G\subset {\rm SL}_2$ of $N$ has the property that $sym^2(G)=A_5$. Hence the action of $G$ on 
  $\mathbb{P}^1$ is that  of $A_5$ and so $G=A_5^{SL_2}$. \\

\noindent 
 {\it Observation}.  The equivalence of Tannaka categories ${\rm Diff}_{\overline{k}/k}\rightarrow Repr_\pi$, where $k=C(z)$ and $\pi=Gal(\overline{k}/k)$, leads to a translation of part (1) of 5.1 into solvability of the embedding problem for $1\rightarrow \{\pm 1\}\rightarrow A_5^{SL_2}\rightarrow A_5\rightarrow 1\ $:\\

\noindent 
{\it Any continuous surjective homomorphism $\pi \rightarrow A_5$ lifts to a continuous surjective homomorphism $\pi \rightarrow A_5^{SL_2}$. }\\   

We note that $k=C(z)$ is known to have this property, see \cite{Ma-Ma}, Theorem 1.10, part (a), on page 272. 

  Let the 1-dimensional module $D$ satisfy $D^{\otimes 2}={\bf 1}$. Then $sym^2(D\otimes N)\cong D^{\otimes 2}\otimes sym^2(N)=M$.
  This proves one implication for part (2) of 5.1. Using the equivalence of Tannaka categories, part (2) translates into:\\
  
  \noindent {\it Let for $i=1,2$ be given surjective continuous homomorphisms $\rho_i:\pi \rightarrow A_5^{SL_2}$ such that 
  $sym^2(\rho_1)\cong sym^2(\rho_2)$. Then there exists a continuous homomorphism $\chi :\pi \rightarrow \{\pm1 \}\subset A_5^{SL_2}$ 
  such that $\rho_1 \cong \rho_2\otimes \chi$.}\\
  
  Let $can: {\rm SL}_2\rightarrow {\rm PSL}_2$ denote the canonical map. The assumption on $\rho_1,\rho_2$  is equivalent to
   $can\circ \rho_1\cong can\circ \rho_2$.  Since the subgroup $A_5\subset {\rm PSL}_2$ is unique up to conjugation, 
  we may suppose $can\circ\rho_1(g)=can\circ\rho_2(g)$ for every $g\in \pi$. Hence $\rho_1(g)=\chi (g)\rho_2(g)$ for some continuous
  homomorphism $\chi:\pi \rightarrow \{\pm 1\}$. \end{proof}

\begin{observations} Comparing operators for $A_5^{SL_2}$ and $A_5$.\\ 
{\rm (1). In \cite{vdP-U} there is a list of the third order differential operators (up to equivalence) with differential Galois group 
$A_5$ and  singular points $0,1,\infty$. The branch types are $[2,3,5], [2,5,5],[3,3,5],[3,5,5](1),[3,5,5](2),[5,5,5]$.
For each branch type there are two differential modules; one for each of the 3-dimensional irreducible characters $\chi_2,\chi _3$.
The genera for the Picard--Vessiot fields are $0,4,5,9,9,13$.

Looking at the genera, one sees that each $A_5$ case is the second symmetric power of two, three or four  second order equations with group 
 $A_5^{SL_2}$ and singularities $0,1,\infty$ (again a list in  \cite{vdP-U}). This is explained by Proposition 5.1 and the observation that there are  three 1-dimensional modules $D$
with $D^{\otimes 2}={\bf 1}$ and singular points $0,1,\infty$. Namely $D=C(z)e$ with $\partial e=ae$ and 
  $a\in \{\frac{1}{2z},\frac{1}{2(z-1)}, \frac{1}{2z(z-1)} \}$. \\ 
  
\noindent  (2).  Comparing the local exponents for $A_5^{SL_2}$ and $A_5$ in both lists of \cite{vdP-U} one sees that only for the two cases of [3,3,5] the operator $L_3$ is a second symmetric
 power. In all other cases the module $M$ is a $sym^2(N)$ but this does not hold for the operators. \\

 \noindent (3). Let $L_{st,A_5^{SL_2}}$ denote the standard second order operator for $A_5^{SL_2}$ i.e., say, with local exponents: $1/4, 3/4||1/3,2/3||-2/5,-3/5$. Let $L_{st,A_5}$ denote the second symmetric power of this operator. Then Klein's theorem for order two equations
 with differential Galois group $\in \{D_n^{SL_2},A_4^{SL_2},S_4^{SL_2}, A_5^{SL_2}\}$ extends to  $A_5$ and order three
 differential equations, see Proposition 5.3. $\ $\hfill  $\square$  }\end{observations}

\noindent {\it Invariants and evaluation.}\\
Generators for the ring $C[x,y,z]^{A_5}$ are, according to \cite{Ca}, \\
\begin{footnotesize}
\[F_2=x^2+yz,\  F_6 = 8x^4yz-2x^2y^2z^2-x(y^5+z^5)+y^3z^3;
 F_{10} = 320x^6y^2z^2-160x^4y^3z^3+\] \[20x^2y^4z^4+ 6y^5z^5
  -4x(y^5+z^5)(32x^4-20x^2yz+5y^2z^2)+y^{10}+z^{10}; F_{15}=\cdots\]
There is one relation
 \[  F_{15}^2+1728F_6^5-F_{10}^3-720F_2F_6^3F_{10}
  +80F_2^2F_6F_{10}^2-64F_2^3(-F_{10}F_2+5F_6^2)^2=0\]
 \end{footnotesize} 
  
\noindent {\it The evaluation for the $A_5$-invariant curve $Z\subset \mathbb{P}^2$ given by $F_2=0$}.\\  
The graded algebra for $Z/A_5$ is $A=C[F_6,F_{10},F_{15}]/(F_{15}+1728F_6^5-F_{10}^3)$. Generators for the
field $A_{((0))}$ are $\frac{F_{15}^2}{F_6^5}$ and $\frac{F_{10}^3}{F_6^5}$ and there is one relation
$\frac{F_{15}^2}{F_6^5}+1728-\frac{F_{10}^3}{F_6^5}=0$. This leads to the evaluation
$ev:(F_2,F_6,F_{10},F_{15})\mapsto (0,1,t^{1/3}, (t-1728)^{1/2})$ or $\mapsto (0,t^4(t-1728)^3,t^7(t-1728)^5,t^{10}(t-1728)^8)$.

The third order differential operator deduced from this evaluation has three singular points $0,1728,\infty$. One normalizes
this operator such that the singular points are $0,1,\infty$.  Then conjugates the operator with the function $(t-1)^{-1/2}t^{-1/3}$ in order
to obtain the required local exponents. The resulting operator $L_c$ identifies with   $L_{st,A_5}$ of Observations 5.2.\\

The operator $L_c$  has to be  {\em equivalent} to  one of the two operators of  \cite{vdP-U},  $A_5$ with branch type $[2,3,5]$, namely $L_u$, the one with local data  $-1,-1/2,1/2||-2/3, -1/3,0||6/5,9/5,2||\mu=43/225$. Below are the formulas for
$L_c$,  $\tilde{L}_u$ obtained by $t\mapsto 1-t$ from $L_u$ and the verification of the equivalence (i.e., the two operators
define the same differential module). 
\begin{footnotesize}
\[ L_c=d_t^3+\frac{3(2t-1)}{t(t-1)}d_t^2+\frac{6264t^2-6389t+800}{900 t^2(t-1)^2}d_t+\frac{1728t-989}{1800t^2(t-1)^2} \]


\[\tilde{L}_u=d_t^3+\frac{8t-4}{t(t-1)}d_t^2+\frac{12744t^2-13169t+2000}{900t^2(t-1)^2}d_t+\frac{7776t^2-12683t+4457}{1800t^2(t-1)^3}       \]

\[\tilde{L}_u\cdot ((t^2-t)d_t^2 +(\frac{14t}{5}-\frac{4}{3})d_t+\frac{48t-49}{60(t-1)})= 
( (t^2-t)d_t^2+(\frac{54t}{5}-\frac{16}{3})d_t+ \frac{1440t^2-1453t+280}{60t(t-1)} )\cdot       L_c    \]
\end{footnotesize}


 \begin{proposition} Every third order operator $L$ over $C(z)$  with differential Galois group $A_5$ is equivalent to a weak pullback of $L_{st,A_5}$.\end{proposition}
 Indeed this follows from Proposition 5.1 and Klein's theorem for second order equations with group $A_5^{SL_2}$. \\

\noindent {\it Comparing evaluations for $A_5^{SL_2}$ and $A_5$}.\\
  Let the second order operator $L_2$ have differential Galois group $A_5^{SL_2}$ and Picard--Vessiot field $K^+$. 
  Then  the third order operator $L_3:=sym^2(L_2)$ has differential Galois group $A_5$ and Picard--Vessiot field 
  $K=(K^+)^Z$, where $Z$ is the center of $A_5^{SL_2}$. The evaluation for $L_2$ is deduced from a homomorphism
  $h_1:{C}(z)[X,Y]\rightarrow K^+$ which sends $X,Y$ to a basis of solutions of $L_2$. The evaluation for $L_3$ is
  deduced from a homomorphism $h_2: {C}(z)[X_1,X_2,X_3]\rightarrow K$ which sends $X_1,X_2,X_3$ to a basis of
  solutions for $L_3$. We may suppose that $X_1,X_2,X_3$ are mapped  to $h_1(X)^2,h_1(XY),-h_1(Y^2)$. 
  It follows that  $F_2=X_1X_3+X_2^2$  lies in the kernel of the evaluation  $h_2:{C}[X_1,X_2,X_3]\rightarrow K$.
   Hence the evaluation for $L_3$ is induced by an evaluation for $L_2$.\\

\noindent {\it Other evaluations for $A_5\subset {\rm SL}(V)$ with $\dim V=3$}.\\
In constructing an evaluation which does not map $F_2$ to 0, one can start by mapping $F_2,F_6,F_{10}$ to $1, a,b$ with  $a,b\in C(z)$, not both constant. Then $F_{15}$ is mapped to a certain $w$ (depending on $a,b$) with $w^2\in C(z)^*$. This leads to an evaluation 
$h: (F_2,F_6,F_{10},F_{15})\mapsto (w^2,w^6a,w^{10}b,w^{16})$.  This ``general formula'' for an evaluation can be refined. The evaluation
induces an $A_5$-invariant curve $Z\subset \mathbb{P}(V)$ such that the normalisation of $Z/A_5$ has genus 0. The curve need not be 
irreducible!  

Not every $h$ constructed in this way produces a third order differential operator. Moreover one might obtain 
third order operators with a differential Galois group which is a proper subgroup of $A_5$ (see the example below).\\

 If one wants an $A_5$-invariant irreducible curve $Z\subset \mathbb{P}(V)$ of degree $> 2$ and with $\leq 4$ branch points,  
 then one can consider $a=\lambda \in C$ and $b=z$.   The equation for the curve is $-\lambda F_2^3+F_6=0$.  Then $w^2\in C[z]$ is a polynomial of degree 3. The singular points of the third order operator are contained in the union of 
 $\{0,\infty\}$ and the roots of $w^2$. There are at most 4 branch points if the discriminant of $w^2$ is zero. 
 These cases are:\\

\noindent 
  (1).  $\lambda =1$. The curve   $-F_2^3+F_6=0$   is reducible and the equation factors as $x(x+y+z)(\mbox{ degree } 4)$.
 The associated operator  \begin{scriptsize}
  \[ L = d_t^3+\frac{3(7t^2-147t+676)}{2(t-4)(3t-37)(t-8)}d_t^2+\frac{3(149t^2-3367t+13584)}{100(3t-37)(t-8)(t-4)^2}d_t-
  \frac{3(t-29)}{200(3t-37)(t-8)(t-4)^2)}\]
   \end{scriptsize}
   factors. The two right hand factors are $L_2=d_t^2+\frac{t-6}{t^2-12t+32}d_t-\frac{1}{100(t^2-12t+32)}$ and 
   $L_1=d_t+\frac{1}{2(t-4)}$. A basis of solutions for $L_2$ is  $(t-6+\sqrt{t^2-12t+32})^{1/10}$, $(t-6+\sqrt{t^2-12t+32})^{-1/10}$
   and the differential Galois group is the dihedral group $D_{10}$ (of order 20).
   It can be seen that the solution $\sqrt{t-4}$ of $L_1$ belongs to the Picard--Vessiot field $C(t,(t-6+\sqrt{t^2-12t+32})^{1/10})$
   of $L_2$. We conclude that the differential Galois group of $L$ is the subgroup $D_{10}$ of $A_5$. \\
   
 \noindent  (2). $\lambda =0$. The curve $F_6=0$ has genus 4 and is a Galois covering of $\mathbb{P}^1_z$ with group $A_5$ ramified over 
the points $0,-64,\infty$. The corresponding order three differential equation
\[ d_t^3+ \frac{7t+256}{2t(t+64)}d_t^2+\frac{149t+1024}{100t^2(t+64)}d_t-\frac{1}{200t^2(t+64)}.\]
has group $A_5$.\\

\noindent (3).  $\lambda =\frac{32}{27}$. We give some details for this interesting example.

\noindent 
The curve $-\frac{32}{27}F_2^3+F_6=0$ has genus 0 and has 10 singular points (all over the cyclotomic field $\mathbb{Q}(\zeta_5)$). The curve is parametrized by 
 $[-5s^3:s^6+3s:3s^5-1]$ and has function field $C(s)$. The group $A_5$ acts on this field and moreover permutes the set of 10 double points.
   $C(s)^{A_5}=C(z)$ where $z$ is equal to  
\begin{footnotesize} 
\[ (s^{60}+2388s^{55}+326394s^{50}-8825700s^{45}+117672975s^{40}+83075976s^{35}+380773868s^{30}\]
\[ -83075976s^{25}+117672975s^{20}+ 8825700s^{15}+326394s^{10}-2388s^5+1)\]
\[/288(s^5(s^2+s-1)^5(s^4+2s^3+4s^2+3s+1)^5(s^4-3s^3+4s^2-2s+1)^5).\]  
\end{footnotesize}
The evaluation defines the following third order differential operator 
\begin{scriptsize}
\[ d_z^3+\dfrac{81(567z-4864)}{2(81z-1024)(81z-448)}d_z^2+\dfrac{19683(4023z-46592)}{100(81z-448)(81z-1024)^2}d_z-
\dfrac{531441}{200(81z-448)(81z-1024)^2}\]
\end{scriptsize}
 that has Picard--Vessiot field $C(s)$ and a basis of solutions $(X_1,X_2,X_3)\mapsto (-5s^3, (s^6+3s), (3s^5-1)) \in C(s)^3$. 
The operator $L$ is equivalent to the standard one, but is not itself a second symmetric power.  Indeed, because of the degrees in $s$, there
 is no quadratic relation over  ${C}$ between the three solutions.  Further ${C}(s)\supset {C}(z)$
  has three branch points namely $z=\frac{448}{81}, \frac{1024}{81}, \infty$. The ramification type must be $[2,3,5]$ since the genus is 0.

\begin{footnotesize}

\end{footnotesize}

\end{document}